\documentclass[11pt, oneside]{amsart}
\usepackage[a4paper, margin={2.5cm,2cm}]{geometry}

\author{Renpeng Zheng}
\address{University of Nottingham, Nottingham, England, United Kingdom}
\email{renpeng.zheng@nottingham.ac.uk}
\usepackage[utf8]{inputenc}
\DeclareUnicodeCharacter{211A}{$\mathbb{Q}$}
\DeclareUnicodeCharacter{211D}{$\mathbb{R}$}
\usepackage{lmodern}
\usepackage[T1]{fontenc}
\usepackage[bb=libus]{mathalpha}
\usepackage{amsmath,latexsym,amsfonts}
\usepackage{amssymb}

\usepackage{CJKutf8}
\newcommand{\zh}[1]{\begin{CJK*}{UTF8}{gkai}#1\end{CJK*}}
\usepackage[french, english]{babel}
\usepackage[strict=true]{csquotes}
\babelprovide{chinese}
\setlocalecaption{chinese}{abstract}{\zh{摘要}}
\newenvironment{abstractZH}{
  \begin{otherlanguage}{chinese}
    \begin{abstract}
      \begin{CJK*}{UTF8}{gkai}
}{
      \end{CJK*}
    \end{abstract}
  \end{otherlanguage}
}
\usepackage[normalem]{ulem}
\usepackage{bbm}
\usepackage[x11names, svgnames, rgb]{xcolor}
\usepackage[
  linktoc=page,
  colorlinks,
  citecolor=gray,
  urlcolor=gray,
]{hyperref}
\usepackage{bookmark}
\bookmarksetup{depth=3}
\usepackage{mathtools}
\usepackage{amsthm}
\newtheorem{theorem}{Theorem}[section]
\newtheorem*{theorem*}{Theorem}
\newtheorem{proposition}[theorem]{Proposition}
\newtheorem{corollary}[theorem]{Corollary}
\newtheorem{lemma}[theorem]{Lemma}

{
  \theoremstyle{definition}
  \newtheorem{definition}[theorem]{Definition}
  \newtheorem{example}[theorem]{Example}

  \newtheorem{lemdefn}[theorem]{Lemma-Definition}
  \newtheorem{thmdefn}[theorem]{Theorem-Definition}
}
{
  \theoremstyle{remark}
  \newtheorem{remark}[theorem]{Remark}
  \newtheorem*{remark*}{Remark}
  
}
\newenvironment{sketchproof}{%
  \proof
}{\endproof}
\usepackage[english, capitalize]{cleveref}
\Crefname{claim}{claim}{claims}
\Crefname{conjecture}{conjecture}{conjectures}
\Crefname{question}{question}{questions}
\Crefname{lemdefn}{lemma-definition}{lemmas-definitions}
\Crefname{lemdefn}{Lemma-Definition}{Lemma-Definitions}
\Crefname{thmdefn}{theorem-definition}{theorems-definitions}
\Crefname{thmdefn}{Theorem-Definition}{Theorems-Definitions}
\Crefname{equation}{eq.}{eq.}
\Crefname{equation}{eq.}{eq.}
\usepackage{mdframed}
\newmdenv[
    linecolor=black,
    linewidth=0.2pt,
    skipbelow=10pt,
    topline=false, bottomline=false, rightline=false
]{inline}
\usepackage[
  style=alphabetic,
  maxalphanames=99,
  url=false,
  isbn=false,
  date=year,
  backref=true,
  backrefstyle=three,
]{biblatex}
\DefineBibliographyStrings{english}{
  backrefpage  = {$\uparrow$},
  backrefpages = {$\uparrow$},
}
\DeclareFieldFormat{eprint:mrnumber}{%
  MR\addspace
  \href{https://mathscinet.ams.org/mathscinet-getitem?mr=#1}{#1}}
\DeclareFieldFormat{eprint:numdam}{%
  Numdam\addcolon\space
  \href{https://www.numdam.org/item/#1}{#1}}
\DeclareFieldFormat{eprint:cup}{%
  \mkbibacro{CUP}\addcolon\space
  \href{https://www.cambridge.org/#1}{#1}}

\usepackage{tikz}
\usetikzlibrary{cd}
\usetikzlibrary{arrows,shapes}
\usetikzlibrary{calc}
\newcommand{\wait}[1]{\textrm{\color{blue}#1}}
\newcommand{\warn}[1]{\textrm{\color{red}#1}}
\newcommand{\new}{\textrm{\color{orange} (NEW)} }
\newcommand{\commentOFF}{
  \renewcommand{\wait}[1]{}
  \renewcommand{\warn}[1]{}
  \renewcommand{\new}{}
}
\makeatletter
\let\over\@@over
\makeatother
\def\mid{\,:\,} 
\def\oparen{\mathopen{}\left\lparen}
\def\cparen{\right\rparen\mathclose{}} 
\def\obrack{\mathopen{}\left\lbrack}
\def\cbrack{\right\rbrack\mathclose{}} 
{
  \catcode`\(=13 \gdef({\oparen} \catcode`\)=13 \gdef){\cparen}
  \catcode`\[=13 \gdef[{\obrack} \catcode`\]=13 \gdef]{\cbrack}
}
\everydisplay\expandafter{\the\everydisplay
  \def\{{\mathopen{}\left\lbrace} \def\}{\right\rbrace\mathclose{}} 
  \catcode`\(=13 \catcode`\)=13 \catcode`\[=13 \catcode`\]=13
}
\newcommand{\<}{\left\langle} \renewcommand{\>}{\right\rangle}

\let\origmaketitle\maketitle
\def\maketitle{
  \begingroup
  \def\uppercasenonmath##1{} 
  \let\MakeUppercase\relax 
  \origmaketitle
  \endgroup
  \vspace{-0.6cm}
}

\makeatletter
\renewcommand{\subsection}{\@startsection{subsection}{2}{\z@}%
  {12pt plus 3pt minus 4pt}{7pt plus 1pt minus 2pt}%
  {\normalfont\bfseries}}
\makeatother

\makeatletter
\renewcommand{\tocsection}[3]{%
  \indentlabel{\@ifnotempty{#2}{\ignorespaces#1 #2.~}}#3}
\renewcommand{\tocsubsection}[3]{\color{black!60}%
  \indentlabel{\@ifnotempty{#2}{\ignorespaces#1 #2.~}}#3}
\renewcommand{\tocsubsubsection}[3]{\color{black!60} 
  \indentlabel{\@ifnotempty{#2}{\ignorespaces#1 \phantom{1.1.}~}}#3}

\def\@tocline#1#2#3#4#5#6#7{\relax
  \ifnum #1>\c@tocdepth 
  \else
    \def\author##1{\newline\textsc{##1}}%
    \par \addpenalty\@secpenalty\addvspace{#2}%
    \begingroup
      \hyphenpenalty\@M
      \@ifempty{#4}{%
        \@tempdima\csname r@tocindent\number#1\endcsname\relax
      }{%
        \@tempdima#4\relax
      }%
      \parindent\z@ \leftskip#3\relax \advance\leftskip\@tempdima\relax
      \rightskip\@pnumwidth plus.2\hsize \parfillskip-\@pnumwidth
      #5\leavevmode\hskip-\@tempdima #6\nobreak\relax
      ~ \dotfill  
      \hbox to\@pnumwidth{\@tocpagenum{#7}}\par
      \nobreak
    \endgroup
  \fi
}
\def\l@subsection{\@tocline{2}{0pt}{1pc}{5pc}{}}
\def\l@subsubsection{\@tocline{3}{0pt}{1pc}{5pc}{}}
\makeatother
\let\oldbigoplus\bigoplus
\renewcommand{\bigoplus}{\oldbigoplus\limits}
\let\oldbigotimes\bigotimes
\renewcommand{\bigotimes}{\oldbigotimes\limits}
\let\oldbigcap\bigcap
\renewcommand{\bigcap}{\oldbigcap\limits}
\let\oldbigcup\bigcup
\renewcommand{\bigcup}{\oldbigcup\limits}
\let\oldbigsqcup\bigsqcup
\renewcommand{\bigsqcup}{\oldbigsqcup\limits}
\let\oldlim\lim
\renewcommand{\lim}{\oldlim\limits}
\let\oldlimsup\limsup
\renewcommand{\limsup}{\oldlimsup\limits}
\let\oldliminf\liminf
\renewcommand{\liminf}{\oldliminf\limits}

\let\oldsup\sup
\renewcommand{\sup}{\oldsup\limits}
\let\oldinf\inf
\renewcommand{\inf}{\oldinf\limits}
\let\oldsum\sum
\renewcommand{\sum}{\oldsum\limits}
\let\oldprod\prod
\renewcommand{\prod}{\oldprod\limits}

\renewcommand{\div}{\bug}
\renewcommand{\emptyset}{\varnothing}
\renewcommand{\epsilon}{\varepsilon}
\renewcommand{\hat}{\widehat}
\newcommand{\surj}{\twoheadrightarrow}

\newcommand{\cl}{\overline}

\newcommand{\Hom}{\mathrm{Hom}}

\newcommand{\Aut}{\mathrm{Aut}}

\newcommand{\rank}{\mathrm{rank}}

\newcommand{\dd}{\mathrm{d}}

\commentOFF

\newcommand{\nocontentsline}[3]{}
\let\origcontentsline\addcontentsline
\newcommand\stoptoc{\let\addcontentsline\nocontentsline}
\newcommand\resumetoc{\let\addcontentsline\origcontentsline}

\title{K-stability of $\mathbb{Q}$-Fano Spherical Varieties via Compatible Divisors}
\date{29th May 2026}

\thanks{This work will appear at Annales de l'Institut Fourier. Copyright (2026) by the Authors. This version is distributed under the CC-BY-ND licence.}

\begin{document}

\maketitle

\begin{abstract}
  We study the K-stability of $\mathbb{Q}$-Fano spherical varieties using compatible divisors. More precisely, if the $\mathbb{Q}$-Fano variety, with a reductive group action, has an open Borel subgroup orbit, then there is a unique anticanonical $\mathbb{Q}$-divisor computing the equivariant stability threshold. This $\mathbb{Q}$-divisor is invariant under the Borel subgroup action, and it characterizes the K-stability of a $\mathbb{Q}$-Fano spherical variety.
\end{abstract}
\begin{abstractZH}
  我们使用了兼容除子来研究$\mathbb{Q}$-Fano的球面簇形的K稳定性。具体而言，如果一个簇形具有一个既约群作用，同时它有一个开的Borel子群轨道，那么唯一存在一个反正则$\mathbb{Q}$-除子，可以被用来计算等变稳定性阈值。这个$\mathbb{Q}$-除子在Borel子群作用下不变，而且它刻画了$\mathbb{Q}$-Fano的球面簇形的K稳定性。
\end{abstractZH}
\begin{otherlanguage}{french}
\begin{abstract}
  Nous étudions la K-stabilité des variétés sphériques de $\mathbb{Q}$-Fano en utilisant des diviseurs compatibles. Plus précisément, si la variété de $\mathbb{Q}$-Fano, munie de l'action d'un groupe réductif, possède une orbite ouverte sous l'action d'un sous-groupe de Borel, il existe alors un unique $\mathbb{Q}$-diviseur anticanonique calculant le seuil de stabilité équivariante. Ce $\mathbb{Q}$-diviseur est invariant sous l'action du sous-groupe de Borel, et il caractérise la K-stabilité d'une variété sphérique de $\mathbb{Q}$-Fano.
\end{abstract}
\end{otherlanguage}

\tableofcontents

\section*{Introduction}\label{sec:intro}
\begingroup
  \renewcommand{\thetheorem}{(\Alph{theorem})}
  \stoptoc
  \newcommand{\ord}{\mathrm{ord}}

\newcommand{\GDValo}{\mathrm{DivVal}^{G, \circ}}
\newcommand{\cano}{\mathrm{K}}

\newcommand{\lct}{\mathrm{lct}}

For a $\mathbb{Q}$-Fano variety $X$, K-stability provides a criterion for the existence of Kähler--Einstein metrics, which is used to be known as the Fano Yau-Tian-Donaldson conjecture.
It is proved in \cite{CDS_KE_Fano} for smooth varieties (Fano manifolds), and in \cite{Berman_Kps_QFano_KE, LXZ_fg_val_compu_thre} for $\mathbb{Q}$-Fano varieties. See \cite{Xu_Book} for a detailed study.

The original definition of K-stability is through (Donaldson's) Futaki invariant on test configurations, but the invention of the so-called stability threshold (in \cite{FO_Kst_Kano,BJ_thre_val_kst}) is crucial in further developments.

Verifying K-stability (or estimating the stability thresholds) for an arbitrary $\mathbb{Q}$-Fano variety was challenging until the Abban--Zhuang method appeared in \cite{AZ_Kstab_fano_flag}. However, several classes of varieties were studied prior to that due to their abundance of symmetries, e.g., toric and spherical varieties.

In this paper, we use the idea of the compatible divisors from \cite{AZ_Kstab_fano_flag} to give an alternative and simpler proof for the formula of equivariant stability threshold for toric and spherical varieties.

In \cite{BJ_thre_val_kst}, the authors show that, for a $\mathbb{Q}$-Fano toric variety $X$ with an algebraic torus $T \subseteq X$, the stability threshold $\delta(X)$ (see \Cref{def:delta}) is equal to the log canonical threshold of a $T$-invariant anticanonical $\mathbb{Q}$-divisor, which we denote by $D^T_X$.
Their proof consists of two parts, first prove the equality $\delta(X) = \delta_T(X)$, where $\delta_T(X)$ denotes the $T$-equivariant stability threshold (see \Cref{def:deltaG}), and then get a formula of the threshold $\delta_T(X)$.
The former part simply follows from Zhuang's equivariant criterion \cite{Zhuang_equi_Kst} (see \Cref{lem:toric:delta}); the latter part is simplified to a short argument using compatible divisors (see \Cref{sec:comp:tor}).


A direct corollary is a rephrasing of the classical criterion for toric varieties,
\begin{corollary}[c.f., \Cref{cor:tor:Kss}]
	The $\mathbb{Q}$-Fano variety $X$ is K-semistable iff the $\mathbb{Q}$-divisor $D^T_X = D_1 + \dots + D_k$, where $D_1, \dots, D_k$ are the $T$-invariant prime divisors on $X$. That is, the $\mathbb{Q}$-divisor $D^T_X$ is identical to the standard choice of the anticanonical divisor for a toric variety.
\end{corollary}

We then find the analogy of the $\mathbb{Q}$-divisor $D^T_X$ on spherical varieties (see \Cref{sec:prelim:sph}). Spherical varieties are the natural generalization of toric varieties, flag varieties and symmetric varieties.

Given a connected reductive group $G$ and a Borel subgroup $B$ of group $G$, a spherical $G$-variety is a $G$-variety that admits an open $B$-orbit. See, for example, \cite{Timashev_Homo_space_equivar_embed} for more properties of spherical varieties. Our main result is to find the $B$-invariant anticanonical $\mathbb{Q}$-divisor $D^B_X$ that computes the $G$-equivariant stability threshold $\delta_G(X)$ (see \Cref{sec:prelim:kstab:GKst}) for a $\mathbb{Q}$-Fano spherical $G$-variety $X$.

\begin{theorem}[c.f., \Cref{thm:sph}]
	Let $X$ be a $\mathbb{Q}$-Fano spherical $G$-variety, the $G$-equivariant stability threshold of $X$ satisfies
	\begin{equation*}
		\delta_G(X) = \inf_{v \in \GDValo_X} {A_X(v) \over v(D^B_X)} = \min_{i = 1, \dots, k} {A_X(v_i) \over v_i(D^B_X)}
	\end{equation*}
	for some $B$-invariant effective anticanonical $\mathbb{Q}$-divisor $D^B_X$ (which is uniquely determined by the variety $X$), where 
	\begin{itemize}
		\item $\GDValo_X$ is the set of non-zero $G$-invariant divisorial valuations on $X$ (see \Cref{sec:prelim:bir:val});
		\item $A_X(v)$ is the log discrepancy of $v$ (see \Cref{sec:prelim:bir:dis}); and
		\item $\{v_i\}_i$ is a prescribed finite subset of the $G$-invariant divisorial valuations on $X$.
	\end{itemize}
\end{theorem}
\begin{remark*}
	The second equality follows from the piecewise $\mathbb{R}_{\geq 0}$-linearity of log discrepancy $A_X$, see \cite{Golota_Delta_Fano_large_aut} (based on \cite[Proposition 5.2]{Pas_sing_sphe}) for the details.
\end{remark*}

As mentioned earlier, by \cite{Zhuang_equi_Kst}, a $\mathbb{Q}$-Fano variety with a group action is K-semistable iff it is equivariantly K-semistable. Therefore, the $\mathbb{Q}$-divisor $D^B_X$ is the key to study the K-stability of spherical varieties.

There have been already several studies on the K-stability of $\mathbb{Q}$-Fano spherical varieties, via geometric analysis (\cite{Delcroix_Kst_Fano_spher}), test configurations (\cite{LLW_wei_Kst_Fano_sph}), and valuative invariants (\cite{Yin_valu_lin_sph}). Our contribution is simplifying \cite[Theorem 4.1]{Yin_valu_lin_sph}, and our formula has the advantage of being free from an \textit{a priori} choice of a $G$-linearization (see \Cref{def:linz}).

\subsection*{Acknowledgements}
I would like to thank Hamid \textsc{Abban} and Johannes \textsc{Hofscheier} for their invaluable guidance and support throughout the study of K-stability and spherical geometry. I am also grateful to Thibaut \textsc{Delcroix} and \textsc{Yin} Chenxi for helpful discussions during the initial stage of this project. Furthermore, the careful reading and constructive comments from the anonymous referee greatly improved the manuscript.

\subsection*{Convention}
Throughout this paper, we work over the field of complex numbers $\mathbb{C}$. A \emph{variety} means a reduced, separated scheme of finite type over $\mathbb{C}$, and we further require them to be irreducible and normal. In particular, all algebraic groups are assumed to be connected.
A \emph{point} of a variety is a closed point (i.e., $\mathbb{C}$-point).

  \resumetoc
\endgroup

\section{Preliminaries}\label{sec:prelim}
\newcommand{\id}{\textrm{id}}
\newcommand{\stsh}{\mathcal{O}}
\renewcommand{\div}{\mathrm{div}}
\newcommand{\vol}{\mathrm{vol}}

\newcommand{\Div}{\mathrm{Div}}
\newcommand{\CDiv}{\mathrm{CDiv}}

\newcommand{\ord}{\mathrm{ord}}

\newcommand{\bc}{\mathrm{bar}}

\newcommand{\Pic}{\mathrm{Pic}}

\subsection{Birational Geometry}\label{sec:prelim:bir}

Given a $\mathbb{Q}$-line bundle $L$, let $\mathbb{N}(L) := \{m \in \mathbb{N} \mid m L \text{ is a line bundle}\}$.

\subsubsection{Valuations}\label{sec:prelim:bir:val}
\newcommand{\Val}{\mathrm{Val}}
\newcommand{\DVal}{\mathrm{DivVal}}

Given a variety $X$, a \emph{($\mathbb{Q}$-)valuation} is a map $v : \mathbb{C}(X)^\times \to \mathbb{Q}$ that satisfies, for all rational functions $f_1, f_2 \in \mathbb{C}(X)$,
\begin{itemize}
	\item $v(\mathbb{C}^\times) = 0$ (and as a convention, $v(0) = +\infty$);
	\item $v(f_1 + f_2) \geq \min(v(f_1), v(f_2))$; and
	\item $v(f_1 \cdot f_2) = v(f_1) + v(f_2)$.
\end{itemize}

We denote the \emph{center} of $v$ (if it exists) by $C_X(v) \subseteq X$. The center of $v$ is a closed irreducible subvariety of $X$ satisfying, for every rational function $f \in \mathbb{C}(X)$, $v(f) > 0$ iff $f = 0$ on $C_X(v)$. The uniqueness of center is guaranteed by separatedness of varieties.

A valuation $v$ is called a \emph{valuation on $X$} if the center exists. The trivial valuation on $X$ always exists, whose center is $X$ itself. If the variety $X$ is complete (e.g., projective), every valuation on $\mathbb{C}(X)$ is a valuation on $X$. We denote all the valuations on $X$ by $\Val_X$.

A Cartier divisor $D$ is locally defined by a single rational function, so we can define $v(D) \in \mathbb{Z}$; similarly, for a $\mathbb{Q}$-Cartier $\mathbb{Q}$-divisor $D$, we have $v(D) \in \mathbb{Q}$. If $D$ is effective, we have the inequality $v(D) \geq 0$. Valuations are $\mathbb{Q}_{\geq 0}$-linear on effective $\mathbb{Q}$-Cartier $\mathbb{Q}$-divisors.

Given a birational morphism $\pi: Y \to X$, a \emph{prime divisor $E$ over $X$} is defined as a prime divisor of $Y$. The prime divisor $E$ defines a valuation $\ord_E$ on $X$, whose center is the image $\pi(E)$. A \emph{divisorial ($\mathbb{Q}$-)valuation} $v$ is of the form $c \: \ord_E$ (for $c \in \mathbb{Q}_{\geq 0}$), and the space of divisorial valuations on $X$ is denoted by $\DVal_X$.

For a variety $X$ with an algebraic group $G$ acting on, a valuation $v$ is \emph{$G$-invariant} if it satisfies $v(g \cdot f) = v(f)$ for every $g \in G$ and $f \in \mathbb{C}(X)$. The space of $G$-invariant valuations is denoted by $\Val^G_X$, and similarly we have the space of $G$-invariant divisorial valuations $\DVal^G_X$.

\subsubsection{Log Discrepancy}\label{sec:prelim:bir:dis}
\newcommand{\cano}{K} 
\newcommand{\DValo}{\DVal^\circ}
\newcommand{\lct}{\mathrm{lct}}

In this section, given a projective variety $X$ with a canonical divisor $\cano_X$ (which always exists as $X$ is assumed to be normal). The variety $X$ is called \emph{$\mathbb{Q}$-Gorenstein} if $\cano_X$ is $\mathbb{Q}$-Cartier.

Assume $X$ is $\mathbb{Q}$-Gorenstein. For every prime divisor $E$ on a birational model $\pi: Y \to X$, we can define the \emph{log discrepancy} of divisorial valuations $c \: \ord_E$ on $X$ as follows
\begin{equation}
	A_X(c \: \ord_E) := c \cdot (1 + \ord_E(\cano_Y - \pi^* \cano_X)) ,
\end{equation}
where $\cano_Y$ is a canonical divisor of $Y$.

\begin{remark}
	For the definition of log discrepancy for general $\mathbb{R}$-valuations, see \cite[\S5]{JM_val_asym_inva_seq_i}.
\end{remark}

A $\mathbb{Q}$-Gorenstein variety $X$ is called \emph{Kawamata log terminal} (in short, \emph{klt}) if $A_X(\ord_E) > 0$ for every divisor $E$ over $X$.

\begin{definition}
	Given a klt variety $X$, for a non-zero effective $\mathbb{Q}$-Cartier $\mathbb{Q}$-divisor $D$ on $X$, the \emph{log canonical threshold} of $D$ is defined by,
	\begin{equation*}
		\lct(D) := \lct(X; D) = \inf_{v \in \DValo_{X}} {A_X(v) \over v(D)} ;
	\end{equation*}
	where $\DValo_{X} : = \{v \in \DVal_X \mid v \neq 0 \}$.
\end{definition}

\subsection{K-stability}\label{sec:prelim:kstab}
\newcommand{\dual}[1]{{#1}^{\lor}}
\newcommand{\gr}{\mathrm{gr}}

\newcommand{\ls}[1]{\mathopen{}\left|#1\right|\mathclose{}}

A klt variety $X$ is called \emph{(klt) $\mathbb{Q}$-Fano} if the anticanonical divisor $-\cano_X$ is ample.  In this section, we will recall some techniques developed for K-stability of $\mathbb{Q}$-Fano varieties.

Consider a klt variety $X$ and an ample $\mathbb{Q}$-line bundle $L$,
\begin{definition}\label[definition]{def:mbtdiv}
	For every $m \in \mathbb{N}(L) \setminus \{0\}$, an \emph{$m$-basis type $\mathbb{Q}$-divisor} of $L$ is a $\mathbb{Q}$-divisor given by
	\begin{equation*}
		D = {1 \over m N_m} \sum_{j = 1}^{N_m} \div(s_j)
	\end{equation*}
	where $\{s_j\}_{j = 1}^{N_m}$ is a basis of $H^0(X, m L)$.
\end{definition}
We remark that $D$ is an effective $\mathbb{Q}$-Cartier $\mathbb{Q}$-divisor $\mathbb{Q}$-linearly equivalent to $L$.

\begin{lemdefn}[c.f., {\cite{FO_Kst_Kano, BJ_thre_val_kst}}]\label[lemdefn]{def:delta}
	For every $m \in \mathbb{N}(L) \setminus \{0\}$, we define the \emph{$m$-th quantized stability threshold}
	\begin{equation*}
		\delta_m(L) := \inf_D \lct(D),
	\end{equation*}
	where $D$ runs through all the $m$-basis type $\mathbb{Q}$-divisors of $L$.
	When $m \to \infty$, the limit of $\delta_m(L)$ exists, denoted by $\delta(L)$, and is called the \emph{stability threshold} (or \emph{delta invariant}) of $L$.

	If $X$ is $\mathbb{Q}$-Fano and $L$ is the anticanonical $\mathbb{Q}$-line bundle, we write $\delta(X) := \delta(L)$ for simplicity.
\end{lemdefn}

Another way to compute stability thresholds is via (divisorial) valuations. 

\begin{lemdefn}[c.f., {\cite[\S2.5]{BJ_thre_val_kst}}]\label[lemdefn]{def:S}
	Given a $v \in \DVal_X$, for every $m \in \mathbb{N}(L) \setminus \{0\}$, we define the invariant $S_m(L; v) := \sup_D v(D)$, where the $\mathbb{Q}$-divisor $D$ runs through all $m$-basis type $\mathbb{Q}$-divisors of $L$.

	When $m \to \infty$, the limit of sequence $(S_m(L; v))_m$ exists, denoted by $S(L; v)$, which is called \emph{expected vanishing order} of $v$.
\end{lemdefn}

\begin{proposition}[c.f., {\cite[\S4.2]{BJ_thre_val_kst}}]\label[proposition]{prop:deltaBYval}
	\begin{equation*}
		\delta_m(L) = \inf_{v \in \DValo_X} {A(v) \over S_m(L; v)}, \quad \delta(L) = \inf_{v \in \DValo_X} {A(v) \over S(L; v)} .
	\end{equation*}
\end{proposition}
\begin{remark*}
	We can always find a divisorial valuation $v \in \DValo_X$ that computes the threshold $\delta_m(L)$, i.e., $\delta_m(L) = {A(v) \over S_m(L; v)}$; there might be no divisorial valuations computing the threshold $\delta(L)$, see \cite[\S6]{BJ_thre_val_kst} and \cite{Xu_min_val_QM} for more general statement.
\end{remark*}

\begin{thmdefn}[c.f., {\cite[Theorem 4.13]{Xu_Book}}]
	For a $\mathbb{Q}$-Fano variety $X$,
	\begin{itemize}
		\item $X$ is \emph{(uniformly) K-stable} if $\delta(X) > 1$,
		\item $X$ is \emph{K-semistable} if $\delta(X) \geq 1$,
		\item $X$ is \emph{K-unstable} if $\delta(X) < 1$.
	\end{itemize}
\end{thmdefn}
\begin{remark*}
	If $\Aut(X)$ is discrete, a Fano manifold $X$ admits a Kähler--Einstein metric iff $X$ is $K$-stable. For an ample $\mathbb{Q}$-line bundle $L$ other than $-\cano_X$, the stability threshold $\delta(L)$ characterize the existence of twisted Kähler--Einstein metrics, see \cite{BBJ_variational_YTD} and \cite{RTZ_basis_div_bala_metr}.
\end{remark*}

\subsubsection{Equivariant K-stability}\label{sec:prelim:kstab:GKst}
\newcommand{\GL}{\mathrm{GL}}
\newcommand{\SL}{\mathrm{SL}}
\newcommand{\GDValo}{\DVal^{G, \circ}}

Fix an algebraic group $G$. Let $X$ be a $G$-variety and $L$ be an ample $\mathbb{Q}$-line bundle on $X$,
\begin{definition}\label[definition]{def:deltaG}
	The \emph{$G$-equivariant stability threshold} of $L$ (c.f., \cite[Definition 4.6]{Zhuang_equi_Kst} and \cite[\S4.4]{Xu_Book}) is defined to be
	\begin{equation*}
		\delta_G(L) := \inf_{v \in \GDValo_X} {A_X(v) \over S(L; v)} ;
	\end{equation*}
	if $\GDValo_X = \emptyset$ (i.e., $X$ is homogeneous (a flag variety)), we say $\delta_G(L) = +\infty$.
\end{definition}
\begin{remark*}
	The definition of $\delta_G(L)$ only depends on the $G$-action on $X$. No $G$-action on $L$ is needed.
\end{remark*}

Similarly, for every $m \in \mathbb{N}(L)$, let
\begin{equation}
	\delta_{m, G}(L) := \inf_{v \in \GDValo_X} {A_X(v) \over S_m(L; v)} .
\end{equation}

\begin{remark}
	Notice that here we use a different approach to \Cref{def:delta}. In general, even if $\delta_{m, G}(L) < +\infty$, it is possible that there is no $G$-invariant $m$-basis type $\mathbb{Q}$-divisors of $L$ at all. For example, the special linear group $\SL_2(\mathbb{C})$ acting on the surface $\mathbb{P}^1 \times \mathbb{P}^1$ diagonally; it is easy to see that there is only one $G$-invariant prime divisor. However, a $1$-basis type $\mathbb{Q}$-divisor of $(\mathbb{P}^1 \times \mathbb{P}^1, \stsh(1, 1))$ need at least $4$ prime divisors.
\end{remark}

In general, we have the inequality $\delta_G(L) \geq \delta(L)$, and Zhuang proves,
\begin{theorem}[c.f., {\cite[Theorem 4.1 and Corollary 4.11]{Zhuang_equi_Kst}}]\label[theorem]{thm:Zhuang}
	Given a $\mathbb{Q}$-Fano $G$-variety $X$, if $X$ is K-unstable ($\delta(X) < 1$) then $\delta_G(X) = \delta(X)$; as a corollary, the variety $X$ is K-semistable ($\delta(X) \geq 1$) iff $\delta_G(X) \geq 1$.
\end{theorem}

\subsection{Spherical Geometry}\label{sec:prelim:sph}
\newcommand{\ch}{\mathfrak{X}}

From now on, let $G$ be a (connected) reductive group. We also fix a Borel subgroup $B \subseteq G$ and a maximal torus $T \subseteq B$.

\begin{definition}\label[definition]{def:sph}
	A \emph{spherical $G$-variety} $X$ is a (normal, irreducible) $G$-variety with an open $B$-orbit.
\end{definition}
\begin{remark*}
	Since $X$ has an open $B$-orbit, it has an open $G$-orbit.
\end{remark*}

\begin{example}
	When $G = B = T$ (and we may further require the $T$-action is faithful), the spherical variety $X$ is a \emph{toric variety}.
\end{example}

We write $\ch(A)$ for the \emph{character lattice} of an algebraic group $A$. For a reductive group $G$, the character lattice $\ch(B) \cong \ch(T)$, and they are isomorphic to the lattice $\mathbb{Z}^r$, where $r := \rank(G)$ is the dimension of $T$.


We will use the following notation. Given a finite-dimensional $G$-module $V$, for every character $\lambda \in \ch(B)$,
\begin{itemize}
	\item $V_\lambda$, the \emph{$\lambda$-isotypic subspace}, is the sum of all (isomorphic) simple $G$-submodules of $V$ with the highest weight $\lambda$; since the vector space $V$ is finite-dimensional, we have $\bigoplus_{\lambda \in \ch(B)} V_\lambda = V$;

	\item $V_\lambda^{(B)}$ is the subset of $0$ and \emph{$B$-semi-invariant vectors} in $V_\lambda$, i.e., $V_\lambda^{(B)} = \{v \in V \mid b \cdot v = \lambda(b) v, \forall b \in B\}$; the subset $V_\lambda^{(B)}$ is a vector space, and it is of dimension $1$ iff $V_\lambda$ is simple.
\end{itemize}

For a general reference about the representation theory of reductive groups, see, for example, \cite{MT_Lin_alg_grp_fin_grp_Lie}. 

Given a $G$-variety $X$, the function field $\mathbb{C}(X)$ is an infinite-dimensional $G$-module, but for every $\lambda \in \ch(B)$, we can still define 
\begin{equation*}
	\mathbb{C}(X)_\lambda^{(B)} := \{f \in \mathbb{C}(X) \mid b \cdot f = \lambda(b) f, \forall b \in B\}.
\end{equation*}
And similar to toric varieties, we can define 
\begin{equation*}
	M := \{\lambda \in \mathfrak{X}(B) \mid \mathbb{C}(X)_\lambda^{(B)} \neq 0 \},
\end{equation*}
which is a sub-lattice of $\mathfrak{X}(B)$; its rank is called the \emph{rank} of $X$.

Given a spherical variety $X$, any $B$-invariant rational function on $X$ is constant; equivalently, for every character $\lambda \in M$, we have $\dim(\mathbb{C}(X)_\lambda^{(B)}) = 1$. For every function $f_\lambda \in \mathbb{C}(X)_\lambda^{(B)} \setminus \{0\}$, the divisor $\div(f_\lambda)$ is principal and $B$-invariant; conversely, any principal $B$-invariant divisor is determined by some character $\lambda \in M$.

\subsubsection{\textit{G}-invariant Valuations}\label{sec:prelim:sph:val}
\newcommand{\V}{\mathcal{V}}

Similar to toric varieties, let $N := \Hom(M, \mathbb{Z})$, $N_\mathbb{Q} := \Hom(M, \mathbb{Q})$, and we denote the perfect pairing by $\<\cdot, \cdot\> : M_\mathbb{Q} \times N_\mathbb{Q} \to \mathbb{Q}$.

Given a valuation $v : \mathbb{C}(X)^\times \to \mathbb{Q}$, we can define a map $\rho(v)(\lambda) := v(f_\lambda)$ for every character $\lambda \in M$; we have $\rho(v) \in N_\mathbb{Q}$, i.e.,
\begin{equation}
\begin{aligned}
	\rho: \{\text{valuations of }\mathbb{C}(X)\} \to N_\mathbb{Q}
\end{aligned}
\end{equation}
is well-defined.

\begin{lemdefn}[c.f., e.g., {\cite[\S4]{Brion_var_sph}}]\label[lemdefn]{def:valcone}
	After restriction,
	\begin{equation*}
		\rho : \{G\text{-invariant valuations of }\mathbb{C}(X)\} \to N_\mathbb{Q}
	\end{equation*}
	is injective. Moreover, its image is a co-simplicial convex cone, called \emph{valuation cone} of $X$, denoted by $\V$.
\end{lemdefn}
\begin{remark*}
	In other words, each $G$-invariant valuation of $\mathbb{C}(X)$ is uniquely determined by its values on $B$-semi-invariant rational functions, or equivalently, principal $B$-invariant divisors.
\end{remark*}

As a result, we identify a $G$-invariant valuation $v$ with a vector $\rho(v) \in \V$.

As the cone $\V$ is polyhedral, the following lemma is intuitive.
\begin{lemma}\label[lemma]{lem:sph:DivValEQVal}
	Assume $X$ is a projective spherical $G$-variety, we have $\Val^G_X = \DVal^G_X$.
\end{lemma}
\begin{remark*}
	Recall that, in this paper we assumed valuations are $\mathbb{Q}$-valued. Note that there always exists $G$-invariant $\mathbb{R}$-valuations which are not of the form $c \: \ord_E$ for some prime divisor $E$ over $X$ and $c \in \mathbb{R}_{\geq 0}$; for instance, quasi-monomial ($\mathbb{R}$-)valuations (see, for example, \cite[\S3]{JM_val_asym_inva_seq_i}).
\end{remark*}
\begin{sketchproof}
	According to \cite[Proposition 2, \S2.4]{Brion_var_sph}, there is a toroidal $G$-equivariant birational morphism $\pi: X' \to X$. For every $G$-invariant valuation $v \in \Val^G_X$, there is a subdivision of $\V$ where $\mathbb{Q}_{\geq 0} v \subseteq \V$ is a cone. The subdivision of $\V$ corresponds to birational model $X''$ of $X'$ (thus a birational model of $X$), and the ray $\mathbb{Q}_{\geq 0} v$ defines a prime divisor of $X''$.
\end{sketchproof}

This lemma is stronger than \cite[Proposition 2.19]{Golota_Delta_Fano_large_aut}, which proves the subset $\DVal^G_X$ is dense in the space of real-valued $G$-invariant valuations.

\section{Compatible Divisors in K-stability}\label{sec:comp}

Compatible divisors were first introduced in \cite[Definition 2.7]{AZ_Kstab_fano_flag}, based on which the Abban-Zhuang method is developed to estimate (local) stability thresholds. In this section, we will use this concept to review the K-stability of toric varieties.

\begin{definition}
	A basis $\{s_1, \dots, s_{N_m}\}$ of $H^0(X, m L)$ is called \emph{compatible} with a valuation $v \in \DVal_X$ if for every $s \in H^0(X, m L)$,
	\begin{equation*}
		v(s) = \min_{c_j \neq 0} v(s_j),
	\end{equation*}
	where $s = \sum c_j s_j$, and $c_j \in \mathbb{C}$.
	When the basis $\{s_1, \dots, s_{N_m}\}$ is compatible with $v$, the $\mathbb{Q}$-divisor $D := {1 \over m N_m} \sum_j \div(s_j)$ is called an $m$-basis type $\mathbb{Q}$-divisor \emph{compatible} with $v$.
\end{definition}

The following lemma appears in the proof of \cite[Proposition 3.1]{AZ_Kstab_fano_flag} and is hidden in \cite[\S2.6 and Lemma 3.5]{BJ_thre_val_kst}. We write a short proof here for completeness.
\begin{lemma}\label[lemma]{lem:Smdef}
	For every valuation $v \in \DVal_X$, we have $S_m(L; v) = v(D)$ if the $m$-basis type $\mathbb{Q}$-divisor $D$ is compatible with $v$.
\end{lemma}
\begin{remark*}
	There is always some $m$-basis type $\mathbb{Q}$-divisor compatible with $v$. One can easily construct a compatible basis with the help of the valuative filtration, see \cite[\S3.1]{BJ_thre_val_kst}. Moreover, there always exists some compatible $\mathbb{Q}$-divisor compatible with any two valuations simultaneously, see \cite[Lemma 3.1]{AZ_Kstab_fano_flag}.
\end{remark*}
\begin{proof}
	Assume the basis $\{s_j\}_{j = 1}^{N_m}$ is compatible with $v$, and reorder it so that
	\begin{equation*}
		v(s_1) \leq \dots \leq v(s_{N_m}).
	\end{equation*}
	Given another basis $\{s_j'\}_{j = 1}^{N_m}$, similarly we can order it as above. Assume there exists a section $s_{j_0}'$ such that $v(s_{j_0}') > v(s_{j_0})$, then $v(s_{j_0 + 1}'), \dots, v(s_{N_m}') > v(s_{j_0})$ and we have a linearly independent subset
	\begin{equation*}
		\{s_{j_0}', s_{j_0 + 1}', \dots, s_{N_m}'\} \subseteq \{s \in H^0(X, m L) \mid v(s) > v(s_{j_0})\} \subseteq \<s_{j_0 + 1}, \dots, s_{N_m}\>_\mathbb{C},
	\end{equation*}
	which is impossible. Therefore, for every $j = 1, \dots, N_m$, we have $v(s_j') \leq v(s_j)$, and thus the basis $\{s_j\}_{j = 1}^{N_m}$ achieves the supremum.
\end{proof}

Recall that, to compute $S_m(L; v)$ by definition (see \Cref{def:S}), we need to consider all (and infinitely many) basis of $H^0(X, m L)$. The compatibility provides another way to compute $S_m(L; v)$. As valuations are invariant under (non-zero) scalar multiplication, we can generalize the compatibility from the bases to decompositions of vector spaces.
\begin{definition}\label[definition]{def:comp:decomp}
	For every valuation $v \in \DVal_X$, a decomposition $H^0(X, m L) =  \bigoplus_{J = 1}^k V_J$ (as a vector space) is \emph{compatible} with $v$ if 
	\begin{equation*}
		v(f) = \min_{J = 1, ..., k} v(f_J),
	\end{equation*}
	for every $f = \sum f_J \in \bigoplus_{J = 1}^k V_J$.
\end{definition}

Similar to compatible bases, we can use compatible decompositions to simplify the computation of $S_m(L; v)$: we only need to consider the basis of $H^0(X, mL)$ which is ``compatible'' with the decomposition.

\begin{lemma}\label[lemma]{lem:SmdefDecom}
	Under the assumption of \Cref{def:comp:decomp}, we have $S_m(L; v) = \sup_D v(D)$, where the $\mathbb{Q}$-divisor $D$ runs through all the $m$-basis type $\mathbb{Q}$-divisors defined from bases of the form $\bigcup_{J = 1}^k \mathcal{B}_J \subseteq H^0(X, m L)$ ($\mathcal{B}_J \subseteq V_J$ is a basis); equivalently,
	\begin{equation*}
		S_m(L; v) = {1 \over N_m} \sum_{J = 1}^k N_{m, J} \cdot S_m(V_J; v),
	\end{equation*}
	the (weighted) average, where $N_{m, J} := \dim(V_J)$.
\end{lemma}
\begin{remark*}
	When the subspace $V_J = \mathbb{C} s$, we have $S_m(V_J; v) = {1 \over m} v(s)$, this lemma reduces to \Cref{lem:Smdef}.
\end{remark*}
\begin{proof}
	For every subspace $V_J$, we can find a basis $\mathcal{B}_J \subseteq V_J$ compatible with $v$, and we have
	\begin{equation*}
		S_m(V_J; v) = {1 \over m N_{m, J}} \sum_j v(s_{J, j}),
	\end{equation*}
	then $\bigcup \mathcal{B}_J$ is a basis of $H^0(X, m L)$ compatible with $v$, which means
	\begin{equation*}
		S_m(L; v) = {1 \over m N_m} \sum_{J, j} v(s_{J, j}).
	\end{equation*}
\end{proof}

\subsection{Warm-up Application: Toric Varieties}\label{sec:comp:tor}
\newcommand{\TDValo}{\DVal^{T, \circ}}

In this section, we will find an anticanonical $\mathbb{Q}$-divisor computing the stability threshold of a toric variety. The idea is to prove that this $\mathbb{Q}$-divisor is compatible with all toric valuations. Then, in \Cref{sec:KstSph}, similarly, we will find an anticanonical $\mathbb{Q}$-divisor for spherical varieties.

We will follow the notations and some simple results from \cite[\S7]{BJ_thre_val_kst}, which is based on \cite{Fulton_intro_tor_vari}.

For a toric $T$-variety $X$, let $M := \ch(T)$, the character lattice, and $M_\mathbb{Q}: = M \otimes_\mathbb{Z} \mathbb{Q}$; let $\chi^u \in \mathbb{C}(T) = \mathbb{C}(X)$ be the character of $T$ associated to $u \in M$, thought of which as a rational function on $X$; let $X_1, \dots, X_k$ be the $T$-invariant prime divisors of $X$.

Given a line bundle $L$, we can pick a $T$-invariant divisor $D_0$ linearly equivalent to $L$, and let $P \subseteq M_\mathbb{Q}$ be the convex ($\mathbb{Q}$-)polytope associated to $D_0$.
Now the subspace $H^0(X, m D_0) \subseteq \mathbb{C}(X)$ is a finite-dimensional $T$-submodule, and it can be decomposed into $1$-dimensional $T$-submodules, as
\begin{equation}
	H^0(X, m D_0) = \bigoplus_{u \in m P \cap M} \mathbb{C} \cdot \chi^u;
\end{equation}
we can define the $m$-basis type $\mathbb{Q}$-divisor
\begin{equation*}
	D_m = {1 \over m N_m} \sum_{u \in m P \cap M}( m D_0 + \div(\chi^u) ) = D_0 + {1 \over N_m} \sum {1 \over m} \div(\chi^u).
\end{equation*}
Now, let $D^T_L := D_0 + {1 \over m'} \div(\chi^{m' \bar{u}})$, where $\bar{u} \in P$ is the barycenter and $m' \bar{u} \in M$ for some $m' \in \mathbb{N} \setminus \{0\}$. The $\mathbb{Q}$-divisor $D^T_L$ is the limit of the sequence $(D_m)_m$ in the sense of coefficients; equivalently,
\begin{equation}\label{eq:toric:vDTL}
	v(D^T_L) = \lim_{m \to \infty} v(D_m)
\end{equation}
for every $T$-invariant valuation $v \in \DVal^T_X$.

The $\mathbb{Q}$-divisor $D^T_L$ is independent of the choice of $D_0$ (and thus the $T$-linearization). It is uniquely determined by the ample ($\mathbb{Q}$-)line bundle $L$. When $L$ is the anticanonical $\mathbb{Q}$-line bundle, we write $D^T_X$ for $D^T_{-\cano_X}$.

\begin{theorem}[c.f., {\cite[Theorem F and Corollary 7.4]{BJ_thre_val_kst}}]\label[theorem]{thm:toric}
	\begin{equation*}
	\delta(X) = \inf_{v \in \TDValo_X} {A_X(v) \over v(D^T_L)} = \min_{i = 1, \dots, k} {1 \over a_i},
	\end{equation*}
	where $a_1 X_1 + \dots + a_k X_k = D^T_X$, and $X_1, \dots, X_k$ are the $T$-invariant prime divisors.
\end{theorem}
\begin{remark*}
	By \cite[Proposition 7.3]{BJ_thre_val_kst}, the equality $\lct(D) = \inf_{v \in \TDValo_X} {A_X(v) \over v(D)}$ holds for every $T$-invariant $\mathbb{Q}$-divisor $D$, so we have $\delta(X) = \lct(D^T_X)$. No analogous result exists for computing the log canonical thresholds in the setting of spherical varieties. 

\end{remark*}

The formula of Blum and Jonsson for toric varieties is miraculously simple, and it completely solves the K-(semi)stability problem of toric varieties. We are going to reprove this theorem in two steps. The first step is to use the compatible divisors to prove,
\begin{lemma}\label[lemma]{lem:toric:comp}
	For every $v \in \TDValo_X$ and all sufficiently divisible integer $m$, we have $S(v) = v(D^T_L)$, and thus $\delta_T(L) = \inf_{v \in \TDValo_X} {A_X(v) \over v(D^T_L)}$ (see \Cref{sec:prelim:kstab:GKst}).
\end{lemma}
\begin{proof}
	We only need to prove for every $v \in \DVal^T_X$, the equality $S_m(v) = v(D_m)$ holds, i.e., to prove $D_m$ is compatible with $v$. However, $T$-invariant valuations are exactly ``toric valuations'', which means,
	\begin{equation*}
		v(\sum_{u \in M} c_u \chi^u) = \min_{\substack{u \in M \\ c_u \neq 0}} v(\chi^u).
	\end{equation*}
	This is exactly the definition of compatibility.
\end{proof}

An intermediate result is
\begin{equation}\label{eq:deltaTm}
	\delta_{m, T}(L) = \inf_{v \in \TDValo_X} {A_X(v) \over v(D_m)},
\end{equation}
which we will use to prove the next lemma.

The second step is
\begin{lemma}[c.f., {\cite[Corollary 7.15]{BJ_thre_val_kst}}]\label[lemma]{lem:toric:delta}
	$\delta(X) = \delta_T(X)$.
\end{lemma}
\begin{proof}
	The original proof of Blum and Jonsson deforms the general valuative filtrations into their initial filtrations (see \cite[\S7.4]{BJ_thre_val_kst}). We are going to use Zhuang's theorem (see \Cref{thm:Zhuang}) instead.

	The standard anticanonical divisor is $X_1 + \dots + X_k$ for the toric variety $X$, and the effective $\mathbb{Q}$-divisor $D_m$ is of the form $D_m = a_1 X_1 + \dots + a_k X_k$ for sufficiently divisible integer $m$. There must be an element $a_i \geq 1$ (otherwise, we have $D_m < X_1 + \dots + X_k$), and by \Cref{eq:deltaTm},
	\begin{equation*}
		\delta_{m, T}(X) \leq {A_X(\ord_{X_i}) \over \ord_{X_i}(D_m)} = {1 \over a_i} \leq 1.
	\end{equation*}
	After taking the limit we have $\delta_T(X) \leq 1$. As $\delta(X) \leq \delta_T(X) \leq 1$, we can apply Zhuang's theorem to show $\delta(X) = \delta_T(X)$.
\end{proof}

\begin{proof}[Proof of \Cref{thm:toric}]
	We just combine \Cref{lem:toric:comp} and \Cref{lem:toric:delta}. The last equality is a well known property of toric varieties, see, for example, \cite[Proposition 7.2]{BJ_thre_val_kst}.
\end{proof}

Recall that, a $\mathbb{Q}$-Fano variety $X$ is K-semistable iff $\delta(X) \geq 1$, we have
\begin{corollary}\label[corollary]{cor:tor:Kss}
	A $\mathbb{Q}$-Fano toric variety $X$ is K-semistable iff the $\mathbb{Q}$-divisor $D^T_X = X_1 + \dots + X_k$.
\end{corollary}
\begin{remark*}
	This is just another way of expressing \cite[Corollary 7.17]{BJ_thre_val_kst}.
\end{remark*}

Spherical varieties are a generalization of toric varieties. We are going to construct similar $\mathbb{Q}$-divisors for $\mathbb{Q}$-Fano spherical varieties.

\section{Rational Linear Series for Spherical Varieties}\label{sec:Qls}
\newcommand{\Qls}[1]{\ls{#1}_\mathbb{Q}}
\newcommand{\divQ}{\mathrm{div}_\mathbb{Q}}

To find a ``special'' anticanonical $\mathbb{Q}$-divisors for a spherical variety, we need to construct the ``space of $\mathbb{Q}$-divisors''.

\begin{definition}
	Consider a polarized variety $(X, L)$, the \emph{(complete) $\mathbb{Q}$-linear series} $\Qls{L}$ is defined as the set of effective $\mathbb{Q}$-Cartier $\mathbb{Q}$-divisors $\mathbb{Q}$-linearly equivalent to $L$.
\end{definition}

For every $m \in \mathbb{N}(L) \setminus \{0\}$, we have ${1 \over m} \ls{m L} \subseteq \Qls{L}$, and, by definition, $\Qls{L} = \bigcup {1 \over m} \ls{m L}$. If $m$ divides $m'$, we have ${1 \over m} \ls{m L} \subseteq {1 \over m'} \ls{m' L}$.

The $\mathbb{Q}$-linear series $\Qls{L}$ is a subset of the $\mathbb{Q}$-vector space $\CDiv(X)_\mathbb{Q} := \CDiv(X) \otimes \mathbb{Q}$, the set of $\mathbb{Q}$-Cartier $\mathbb{Q}$-divisors of $X$; moreover, the subset $\Qls{L}$ is \emph{($\mathbb{Q}$-)convex}:
\begin{proposition}
	For every two $\mathbb{Q}$-divisors $D, E \in \Qls{L}$, and coefficient $t \in \mathbb{Q} \cap [0, 1]$, the convex combination $t D + (1 - t) E \in \Qls{L}$.
\end{proposition}
\begin{proof}
	Fix $t = {p \over p + q}$ (where $p, q \in \mathbb{N}, p + q > 0$), $D \in {1 \over m} \ls{m L}, E \in {1 \over m'} \ls{m' L}$, and then we have
	\begin{equation*}
		m m' p D \in \ls{m m' p L}, \quad m m' q E \in \ls{m m' q L}
	\end{equation*}
	then, by summation, we get  $t D + (1 - t) E \in {1 \over m m' (p + q)} \ls{m m' (p + q) L} \subseteq \Qls{L}$.
\end{proof}

Just as $\CDiv(X)_\mathbb{Q}$, the convex subset $\Qls{L}$ is infinite-dimensional. We do not have enough tools to study it.

Luckily, for a toric $T$-variety $X$, the subset of $\Qls{L}$ of $T$-invariant $\mathbb{Q}$-divisors, denoted by $\Qls{L}^T$, is a (finite-dimensional) polytope:
\begin{example}\label[example]{eg:polytope:toric}
	Suppose $L$ is a line bundle. Pick a $T$-invariant divisor $D_0$ linearly equivalent to $L$, and let $P \subseteq M_\mathbb{Q}$ be the convex ($\mathbb{Q}$-)polytope associated to $D_0$, we have the bijection
	\begin{equation*}
	\begin{aligned}
		q: P &\to \Qls{L}^T, \\
		u &\mapsto D_0 + {1 \over m} \div(\chi^{m u}),
	\end{aligned}
	\end{equation*}
	for some $m \in \mathbb{N} \setminus \{0\}$ such that $m u \in M$.
\end{example}

Now, fix a reductive group $G$ and a Borel subgroup $B \subseteq G$. For a polarized variety $(X, L)$ and $G$-action on $X$, we denote the subset of $G$-invariant (resp., $B$-invariant) $\mathbb{Q}$-divisors in $\Qls{L}$ by $\Qls{L}^G$ (resp., $\Qls{L}^B$). Our main interest lies in $\Qls{L}^B$ as it is never empty.

\subsection{Moment Polytopes}\label{sec:Qls:mome}

The natural generalization of the polytopes in \Cref{eg:polytope:toric} is the moment polytopes of line bundles, which depends on the choice of linearizations.

\begin{definition}[c.f., {\cite[\S2]{KKLV_loc_prop_alggrp_act}}]\label[definition]{def:linz}
	Let $X$ be a $G$-variety, and $L$ a line bundle on $X$ with projection $\pi : L \to X$. A \emph{$G$-linearization} $\alpha$ of $L$ is an action of $G$ on $L$ (as a variety), such that $\pi$ is $G$-equivariant and $\alpha$ is linear on each fiber.

	A $G$-linearization $\alpha$ on the line bundle $L$ defines a $G$-linearization on $m L$ for every $m \in \mathbb{N}$. More generally, if $L$ is a $\mathbb{Q}$-line bundle, a \emph{$\mathbb{Q}$-$G$-linearization} $\alpha$ is a $G$-linearization of $m L$ for some $m \in \mathbb{N}(L)$; two $\mathbb{Q}$-$G$-linearizations $\alpha, \alpha'$ of $L$ are the same if they agree on some common multiple $m L$.

	Let $\mathbb{N}(L, \alpha) := \{m \in \mathbb{N}(L) \mid \alpha \text{ is a } G\text{-linearization of the line bundle } m L\}$; then, for every $m \in \mathbb{N}(L, \alpha)$, the vector space $H^0(X, m L)$ is a non-zero $G$-module.
\end{definition}
\begin{remark*}
	By \cite[\S2.4]{KKLV_loc_prop_alggrp_act}, for every line bundle $L$, there is a $G$-linearization for some multiple $m L$ ($X$ need to be normal); equivalently, every $\mathbb{Q}$-line bundle admits a $\mathbb{Q}$-$G$-linearization.
\end{remark*}

From now on, we will not distinguish the $G$-linearizations and the $\mathbb{Q}$-$G$-linearizations.

\begin{lemdefn}[c.f., {\cite[\S2.1]{CFPvS_mome_poly}}, based on \cite{Brion_sur_mome}]\label[lemdefn]{def:mome_poly}
	Given a polarized $G$-variety $(X, L)$ and a $G$-linearization $\alpha$, for every $m \in \mathbb{N}(L, \alpha)$, let
	\begin{equation*}
		P_m(X, L, \alpha) = \{\lambda \in \mathfrak{X}(B) \mid H^0(X, m L)_\lambda^{(B)} \neq \{0\} \}
	\end{equation*}
	(see \Cref{sec:prelim:sph} for the definition of $V_\lambda^{(B)}$), the \emph{moment polytope},
	\begin{equation*}
		P(X, L, \alpha) = \bigcup_{m \in \mathbb{N}(L, \alpha)} {1 \over m} P_m(X, L, \alpha) \subseteq \mathfrak{X}(B) \otimes \mathbb{Q}
	\end{equation*}
	is a convex hull of finitely many points.
\end{lemdefn}

The following result is directly deduced from the definition of moment polytope, 

\begin{lemma}\label[lemma]{lem:moneSPAN}
	For every $\lambda \in P(X, L, \alpha)$, the translation subset $P(X, L, \alpha) - \lambda \subseteq M_\mathbb{Q}$ spans $M_\mathbb{Q}$ (as a $\mathbb{Q}$-vector space). More specifically, when $m \in \mathbb{N}(L, \alpha)$ and line bundle $m L$ is very ample, for every $\lambda \in P_m(X, L, \alpha)$, the subset $P_m(X, L, \alpha) - \lambda$ generates $M$ (as an Abelian group).
\end{lemma}
\begin{proof}
	The first part is a reformulation of \cite[\S1.2, Proposition 3]{Brion_var_sph} (see also \cite[Lemma 2.21 (ii)]{Brion_intro_act_alg_grp}). We now prove the second part.

	For every $m \in \mathbb{N}(L, \alpha)$, we pick a nonzero section
	$s \in H^0(X, mL)^{(B)}_\lambda$, where the $\lambda \in P_m(X, L, \alpha)$. Now, for every
	$\lambda' \in P_m(X, L, \alpha)$, there exists a nonzero section
	$s' \in H^0(X, mL)^{(B)}_{\lambda'}$.
	The ratio of these two sections is a rational function, namely
	$f := s' / s \in \mathbb{C}(X)$, and it satisfies
	\begin{equation*}
		b \cdot f = (b \cdot s') / (b \cdot s)
		= (\lambda' - \lambda)(b)\, f,
	\end{equation*}
	for every $b \in B$. Therefore, by the definition of $M$, we have
	$\lambda' - \lambda \in M$, and thus
	$P_m(X, L, \alpha) - \lambda \subseteq M$.

	Now assume that $mL$ is very ample. This implies that for every
	$\Gamma \in P(X, L, \alpha)$, there exists a sufficiently large integer
	$N$, which is a multiple of $m$, such that
	\begin{equation*}
		N \Gamma = \sum_{{\lambda'} \in P_m(X, L, \alpha)} a_{\lambda'} {\lambda'},
	\end{equation*}
	where $a_{\lambda'} \in \mathbb{N}$ for every
	${\lambda'} \in P_m(X, L, \alpha)$.
	Finally, for every $u \in M$, we can find a sufficiently large $N$
	and $\Gamma, \Gamma' \in P(X, L, \alpha)$ such that
	$u = N(\Gamma - \Gamma')$, which can be further written as a
	$\mathbb{Z}$-linear combination of elements of
	$P_m(X, L, \alpha) - \lambda$.
\end{proof}

\subsection{Convex Projections}\label{sec:Qls:proj}

Now, we are going to study the relationship between the $\mathbb{Q}$-linear series and moment polytopes.

For $m \in \mathbb{N}(L) \setminus \{0\}$, we have the map $\div : H^0(X, m L) \setminus \{0\} \surj \ls{m L}$; by composing it with $\ls{m L} \to {1 \over m} \ls{m L} \subseteq \Qls{L}$, we get a normalized version of $\div$ as
\begin{equation}
\begin{aligned}
	\divQ : H^0(X, m L) \setminus \{0\} &\to \Qls{L} , \\
	s &\mapsto {1 \over m} \div(s) .
\end{aligned}
\end{equation}

After fixing a $G$-linearization $\alpha$, for every $m \in \mathbb{N}(L, \alpha)$, space $H^0(X, m L)$ is a non-zero $G$-module, and we can restrict $\divQ$ to
\begin{equation*}
	\divQ: H^0(X, m L)^{(B)} \setminus \{0\} \surj {1 \over m} \ls{m L}^B .
\end{equation*}

We can define the following map,
\begin{equation}
\begin{aligned}
	p_\alpha: \Qls{L}^B &\to P(X, L, \alpha), \\
	D &\mapsto \lambda / m,
\end{aligned}
\end{equation}
where $D = \divQ(s)$ for some $m \in \mathbb{N}(L, \alpha)$, $\lambda \in \ch(B)$, and $s \in H^0(X, m L)^{(B)}_\lambda$.

\begin{proposition}
	The map $p_\alpha$ is a well-defined $\mathbb{Q}$-convex-linear surjective map. If the variety $X$ is spherical, the map $p_\alpha$ is a convex-linear isomorphism.
\end{proposition}
\begin{proof}
	First, to show it is well-defined (then surjectivity is from the definition of moment polytopes), for every two sections
	\begin{equation*}
		s \in H^0(X, m L)^{(B)}_\lambda, \text{ and } s' \in H^0(X, m' L)^{(B)}_{\lambda'},
	\end{equation*}
	assume $\divQ(s) = \divQ(s')$, we have (up to a coefficient) $s^{m'} = (s')^m$, which means
	\begin{equation*}
		\lambda / m = \lambda' / m' .
	\end{equation*}
	
	Second, to prove it is a $\mathbb{Q}$-convex linear map, consider two sections
	\begin{equation*}
		s \in H^0(X, m L)^{(B)}_\lambda, \quad s' \in H^0(X, m L)^{(B)}_{\lambda'}
	\end{equation*}
	(it is always possible to find a common $m \in \mathbb{N}(L, \alpha)$) and coefficient $t = {p \over p + q} \in \mathbb{Q} \cap [0, 1]$ (where $p, q \in \mathbb{N}, p + q > 0$). Observe that
	\begin{equation*}
		s^p {s'}^q \in H^0(X, (p + q) m L)^{(B)}_{p \lambda + q \lambda'},
	\end{equation*}
	which results in
	\begin{equation*}
	\begin{aligned}
		p_\alpha(t \cdot \divQ(s) + (1 - t) \cdot \divQ(s')) &= p_\alpha(\divQ(s^p {s'}^q)) \\
		&= {p \lambda + q \lambda' \over (p + q) m} = t \cdot {\lambda \over m} + (1 - t) \cdot {\lambda' \over m} \\
		&= t \cdot p_\alpha(\divQ(s)) + (1 - t) \cdot p_\alpha(\divQ(s')).
	\end{aligned}
	\end{equation*}

	Finally, if $X$ is spherical, the space of $B$-semi-invariant sections $H^0(X, m L)^{(B)}_\lambda$ is $1$-dimensional, so we can define the inverse map 
	\begin{equation*}
		q : P(X, L, \alpha) \to \Qls{L}^B ,
	\end{equation*}
	such that, for every $u = \lambda / m \in {1 \over m} P_m(X, L, \alpha)$, there is a unique (up to scaling) section $s \in H^0(X, m L)^{(B)}_\lambda$. Note that $\divQ(s)$ is independent of the choice of $m$. It is straightforward to check the convex-linearity of the map $q$ by repeating the argument above.
\end{proof}

\begin{example}
	For a flag $G$-variety $X$, and an ample $\mathbb{Q}$-line bundle $L$, the $B$-invariant $\mathbb{Q}$-linear series $\Qls{L}^B$ is a singleton. That means, there is a unique effective $B$-invariant $\mathbb{Q}$-divisor $\mathbb{Q}$-linearly equivalent to $L$.
\end{example}

\subsection{Measures on Real Linear Series}\label{sec:Qls:meas}
\newcommand{\Rls}[1]{\ls{#1}_\mathbb{R}}

In \Cref{sec:comp:tor}, we define the $\mathbb{Q}$-divisor $D^T_L$ as a barycenter of the polytope. To find the barycenter of a moment polytope, we need to find a suitable measure on it.

As the character lattice $\ch(B) \cong \mathbb{Z}^{\rank(G)}$ is finite-dimensional, we can denote the closure of the moment ($\mathbb{Q}$-)polytope $P(X, L, \alpha) \subseteq \ch(B)_\mathbb{R} := \ch(B) \otimes_\mathbb{Z} \mathbb{R}$ by $\cl{P_\alpha}$. For every $m \in \mathbb{N}(L, \alpha)$, there is a measure $\dd \mu^\alpha_m$ on $\cl{P_\alpha}$, by
\begin{equation}
	\dd \mu^\alpha_m := {1 \over m^n} \sum_{\lambda \in P_m(X, L, \alpha)} N_{m, \lambda} \delta_{\lambda / m},
\end{equation}
where $\delta_u$ is the Dirac delta measure supported at the point $u \in \cl{P_\alpha}$, $n := \dim(X)$, and $N_{m, \lambda} = \dim(H^0(X, m L)_\lambda)$; $\dd \mu^\alpha_m$ is an atomic measure supported on ${1 \over m} P_m$.

\begin{lemdefn}
	There exists a measure $\dd \mu^\alpha$ on $\cl{P_\alpha}$, which is the weak limit of $\{\dd \mu^\alpha_m\}_{m = 1}^\infty$, and $\int_{\cl{P_\alpha}} \dd \mu^\alpha = \vol(L) / n!$.
\end{lemdefn}
\begin{remark*}
	It is the $G$-equivariant \emph{Duistermaat-Heckman measure} in \cite[Definition-Theorem 2.31]{Xu_Book}'s terminology. See also \cite{Delcroix_Kst_Fano_spher} and \cite[\S4.3]{Yin_valu_lin_sph}.
\end{remark*}
\begin{proof}
	See \cite[\S4.1]{Brion_grp_Pic_num_char_sph}.
\end{proof}

\begin{example}
	For toric varieties, the measure $\dd \mu$ is the lattice measure (or Lebesgue measure, uniform measure).
\end{example}

For a spherical $G$-variety $X$, there are finitely many $B$-invariant prime divisors, so the space of $B$-invariant $\mathbb{Q}$-Cartier $\mathbb{Q}$-divisors, denoted by $\CDiv(X)^B_\mathbb{Q}$, is a finite-dimensional $\mathbb{Q}$-vector space. Then, for every ample $\mathbb{Q}$-line bundle $L$, we can denote the closure of $B$-invariant $\mathbb{Q}$-linear series $\Qls{L}^B$ by $\Rls{L}^B \subseteq \CDiv(X)^B_\mathbb{R} := \CDiv(X)^B_\mathbb{Q} \otimes_\mathbb{Q} \mathbb{R}$.

Now, the extension $p_\alpha: \Rls{L}^B \to \cl{P_\alpha}$ is (still) a convex-linear isomorphism, and we can get the measure $\dd \mu_m$ and $\dd \mu$ on $\Rls{L}^B$.

\begin{proposition}\label[proposition]{prop:measONrls}
	The measures $\dd \mu_m$ (if they are well-defined) and $\dd \mu$ are independent of the choice of $\alpha$.
\end{proposition}
\begin{proof}
	We only need to prove $\dd \mu_m$ is independent of $\alpha$. Consider two $G$-linearization $\alpha, \beta$, and $m \in \mathbb{N}(L, \alpha) \cap \mathbb{N}(L, \beta)$, we have
	\begin{equation*}
		\phi := m \cdot (p_\beta \circ p_\alpha^{-1}) : P_m(X, L, \alpha) \to P_m(X, L, \beta).
	\end{equation*}
	For simplicity, define $G$-modules $V := (H^0(X, m L), \cdot_\alpha), W := (H^0(X, m L), \cdot_\beta)$, and the claim is equivalent to that for every $\lambda \in P_m(X, L, \alpha)$, the equality $\dim(V_\lambda) = \dim(W_{\phi(\lambda)})$ holds.
	
	To begin with, we have $V_\lambda^{(B)} = W_{\phi(\lambda)}^{(B)} = \mathbb{C} s$ for some $s \in H^0(X, m L)$. Let $D := \divQ(s)$, we can see $D \in {1 \over m} \ls{m L}^B$. For every $g \in G$,
	\begin{equation*}
		\divQ(g \cdot_\alpha s) = g \cdot D = \divQ(g \cdot_\beta s);
	\end{equation*}
	therefore, we get $V_\lambda = W_{\phi(\lambda)}$, a stronger result.
\end{proof}


As the space $\Rls{L}^B$ is compact, the barycenter of measures $\dd \mu_m$ and $\dd \mu$ exists, denoted by $\bc_{\dd \mu_m}(\Rls{L}^B)$ and $\bc_{\dd \mu}(\Rls{L}^B)$, respectively. According to \cite[\S4.1]{Brion_grp_Pic_num_char_sph}, we can show $\bc_{\dd \mu}(\Rls{L}^B)$ is indeed a $\mathbb{Q}$-divisor.

\section{K-(semi)stability of ℚ-Fano Spherical Varieties}\label{sec:KstSph}


We are going to show the $\mathbb{Q}$-divisor $D^B_L := \bc_{\dd \mu}(\Rls{L}^B)$ is the spherical analogy of $D^T_L$ in \Cref{sec:comp:tor}. Similar to \Cref{thm:toric}, our main result is
\begin{theorem}\label[theorem]{thm:sph}
	Let $X$ be a spherical $G$-variety, and let $L$ be a $\mathbb{Q}$-line bundle on $X$, the $G$-equivariant stability threshold of $L$ satisfies
	\begin{equation*}
		\delta_G(L) = \inf_{v \in \DVal^G_X} {A_X(v) \over v(D^B_L)}.
	\end{equation*}
\end{theorem}

Given a polarized spherical $G$-variety $(X, L)$ and a $G$-linearization $\alpha$, for every $m \in \mathbb{N}(L, \alpha)$ such that $m L$ is very ample, we can decompose the finite-dimensional $G$-module $H^0(X, m L)$ into
\begin{equation}\label{eq:decompSph}
	H^0(X, m L) = \bigoplus_{\lambda \in P_m(X, L, \alpha)} H^0(X, m L)_\lambda .
\end{equation}
For simplicity, let $V := H^0(X, m L)$ (thus $V_\lambda = H^0(X, m L)_\lambda$), and $N_{m, \lambda} := \dim(V_\lambda)$.

For every character $\lambda \in P_m(X, L, \alpha)$, we have a unique (up to scaling) non-zero section $e^B_\lambda \in V^{(B)}_\lambda$.

\begin{lemma}\label[lemma]{lem:GvalBYeb}
	Given two valuations $v, v' \in \DVal^G_X$, if for every character $\lambda \in P_m(X, L, \alpha)$, we have $v(e^B_\lambda) = v'(e^B_\lambda)$, then $v = v'$.
\end{lemma}
\begin{proof}
	It is a simple corollary of \Cref{lem:moneSPAN} and \Cref{def:valcone}. Note that we have assumed $m L$ to be very ample.
\end{proof}

\begin{lemma}\label[lemma]{lem:GvalCONST}
	Every valuation $v \in \DVal^G_X$ is constant on $V_\lambda \setminus \{0\}$ for every character $\lambda \in P_m(X, L, \alpha)$; more precisely, the equality $v(s) = v(e^B_\lambda)$ holds for every section $s \in V_\lambda \setminus \{0\}$.
\end{lemma}
\begin{proof}
	For every $v \in \DVal^G_X$, we can define a $G$-invariant valuation
	\begin{equation*}
		\hat{v}(s) := \min\limits_{s_\lambda \neq 0} v(e^B_\lambda)
	\end{equation*}
	where $s = \sum_{\lambda \in P_m(X, L, \alpha)} s_\lambda \in \bigoplus_{\lambda \in P_m(X, L, \alpha)} V_\lambda$. By \Cref{lem:GvalBYeb}, we have $v = \hat{v}$.

	Therefore, the equality $v(e^B_\lambda) = \hat{v}(e^B_\lambda) = \hat{v}(s) = v(s)$ holds for every nonzero section $s \in V_\lambda \setminus \{0\}$.
\end{proof}

\begin{lemma}\label[lemma]{lem:GvalCPTdecomp}
	The decomposition \Cref{eq:decompSph} is compatible (see \Cref{sec:comp}) with any valuation $v \in \DVal^G_X$.
\end{lemma}
\begin{proof}
	According to the proof of \Cref{lem:GvalCONST}, every $v \in \DVal^G_X$ satisfies
	\begin{equation*}
		v(s) = \min_{s_\lambda \neq 0} v(s_\lambda)
	\end{equation*}
	where $s = \sum s_\lambda$.
\end{proof}

We are now well-prepared to prove the main theorem.

\begin{proof}[Proof of \Cref{thm:sph}]
	For every character $\lambda \in P_m(X, L, \alpha)$, we pick a basis $\{s_{\lambda, j}\}_j \subseteq V_\lambda$. For every valuation $v \in \DVal^G_X$, by \Cref{lem:GvalCONST}, the valuation $v$ is constant on $V_\lambda$, so the basis $\{s_{\lambda, j}\}_j$ is compatible with $v$ on $V_\lambda$. Now $\{s_{\lambda, j}\}_{\lambda, j}$ is a basis of $V_m$ which is compatible with every $v \in \DVal^G_X$, and we denote the corresponding $m$-basis type $\mathbb{Q}$-divisor
	\begin{equation*}
		D_m = {1 \over N_m} \sum_{\lambda \in P_m(X, L, \alpha)} \sum_{j = 1}^{N_{m, \lambda}} \divQ(s_{\lambda, j}).
	\end{equation*}
	Consequently, by \Cref{lem:Smdef}, we have,
	\begin{equation*}
		\delta_{m, G}(L) = \inf_{v \in \GDValo_X} {A_X(v) \over v(D_m)}.
	\end{equation*}

	It is challenging to find the limit of $(D_m)_m$, as $\Qls{L}$ is infinite-dimensional. That is why we introduce the $B$-invariant $\mathbb{Q}$-divisor  $D^B_m := {1 \over N_m} \sum_\lambda N_{m, \lambda} \divQ(e^B_\lambda) = \bc_{\mu_m}(\Rls{L}^B)$. We have,
	\begin{equation*}
		v(D_m) = {1 \over m N_m} \sum_\lambda \sum_{j = 1}^{N_{m, \lambda}} v(s_{\lambda, j}) = {1 \over m N_m} \sum_\lambda N_{m, \lambda} \: v(e^B_\lambda) = v(D^B_m).
	\end{equation*}
	The sequence $(D^B_m)_m$ is convergent to $D^B_L = \bc_{\mu}(\Rls{L}^B)$ in the finite dimensional topological space $\Rls{L}^B$. As $v$ is linear (thus continuous) on $\Rls{L}^B$, we finally get,
	\begin{equation*}
		\delta_G(L) = \inf_{v \in \GDValo_X} {A_X(v) \over v(D^B_L)}.
	\end{equation*}
\end{proof}

\subsection{Comparison with Brion's \textit{B}-invariant anticanonical divisor}\label{sec:KstSph:Brion}

Similar to toric varieties, for a spherical $G$-variety $X$, there are only finitely many $B$-invariant prime divisors, see \cite[Theorem 2.2.2]{Pezzini_lec_sph_wond_vari}. Let
\begin{itemize}
    \item $X_1, \dots, X_k$ be the $G$-invariant prime divisors on $X$, and
    \item $Y_1, \dots, Y_{k'}$ be the $B$-invariant but not $G$-invariant prime divisors on $X$.
\end{itemize}
\begin{remark}
    The prime divisors $Y_1, \dots, Y_{k'}$ are called \emph{colors} of a spherical variety.
\end{remark}

As the $\mathbb{Q}$-divisor $D^B_X$ is $B$-invariant, we can find the coefficients of $D^B_X$ by 
\begin{equation}
    D^B_X = \sum_{i = 1}^{k'} \alpha_i Y_i + \sum_{j = 1}^{k} \beta_j X_j
\end{equation}
where $\alpha_i, \beta_j \in \mathbb{Q}_{\geq 0}$.

Similar to \Cref{cor:tor:Kss}, we have a weaker result for spherical varieties,
\begin{corollary}\label[corollary]{cor:sph:Kss}
	If a $\mathbb{Q}$-Fano spherical variety $X$ is K-semistable, then we have the inequality $\beta_j \leq 1$ for every $j = 1, \dots, k$.
\end{corollary}
\begin{proof}
    This follows directly from \Cref{thm:sph} and \Cref{thm:Zhuang}. For every $j$, the coefficient $\beta_j = \ord_{X_j}(D^B_X)$, and as divisors $X_1, \dots, X_k$ are prime divisors on $X$, we have $A_X(\ord_{X_j}) = 1$.
\end{proof}

There is another natural $B$-invariant representative for anticanonical $\mathbb{Q}$-divisors, which is indeed a $\mathbb{Q}$-Cartier divisor, for a (klt) Fano spherical $G$-variety $X$.
\begin{theorem}[c.f., {\cite[Theorem 1.2]{GH_gor_sph_fano_var}}, based on {\cite[\S4.1, 4.2]{Brion_cur_div_sph_var}}]
    There is a $B$-invariant anticanonical divisor of the form
    \begin{equation*}
        -\cano_X := \sum_{i = 1}^{k'} m_i Y_i + \sum_{j = 1}^{k} X_j,
    \end{equation*}
    where the coefficients $m_j \in \mathbb{Z}_{> 0}$ depend only on the open $G$-orbit of variety $X$.
\end{theorem}
\begin{remark*}
    We call this divisor \emph{Brion's anticanonical divisor}. For toric varieties, Brion's anticanonical divisor is the same as the standard anticanonical divisor $X_1 + \dots + X_k$.
\end{remark*}

Following \Cref{sec:prelim:bir:dis}, we can use Brion's anticanonical divisor to define the log discrepancy.

There is a deep relationship between Brion's anticanonical divisor and the $\mathbb{Q}$-divisor $D^B_X$ for K-semistable varieties.
\begin{corollary}
    Given a $\mathbb{Q}$-Fano spherical $G$-variety $X$, $X$ is K-semistable iff for every valuation $v \in \GDValo_X$, the equality $v(D^B_X) \leq v(-\cano_X)$ holds. Specifically, if variety $X$ is horospherical, which means the valuation cone is the whole space, then $X$ is K-semistable iff $D^B_X = -\cano_X$.
\end{corollary}
\begin{remark*}
    This is the reformulation of Delcroix's criteria (see \cite[Theorem 5.3 and Corollary 5.7]{Delcroix_Kst_Fano_spher}). Note that the anticanonical divisor $-\cano_X \in \Qls{L}^B$ here plays the same role as the point $2 \rho_Q \in \Delta^+$ in \textit{loc.~cit.}
\end{remark*}
\begin{proof}
    According to \cite[Proof of Proposition 5.2]{Pas_sing_sphe}, for every valuation $v \in \GDValo$, we have $v(-\cano_X) = A_X(v)$. The result then follows directly from \Cref{thm:sph}.
\end{proof}

Finally, let us study the coefficient $\alpha_i$ of the $\mathbb{Q}$-divisor $D^B_X$. If the image of valuation $\ord_{Y_i}$ is in the valuation cone, then we have the inequality $\alpha_i \leq m_i$. In general, coefficient $\alpha_i$ may be larger than $m_i$.
\begin{example}
    Let $G := \SL_2(\mathbb{C})$, and the Borel subgroup $B$ is the subgroup of upper triangular matrices. The variety $\mathbb{P}^1 \times \mathbb{P}^1$ is spherical under the diagonal action. The variety $X$ has three $B$-invariant prime divisors,
    \begin{equation*}
    \begin{aligned}
        Y_1 &:= \{[1: 0]\} \times \mathbb{P}^1, \qquad Y_2 := \mathbb{P}^1 \times \{[1: 0]\}, \\
        X_1 &:= \{(x, x) \in \mathbb{P}^1 \times \mathbb{P}^1 \mid x \in \mathbb{P}^1 \}.
    \end{aligned}
    \end{equation*}
    The anticanonical bundle is $\stsh_X(2, 2)$, so that $-\cano_X = Y_1 + Y_2 + X_1$, and according to \cite[Example 5.10]{Delcroix_Kst_Fano_spher}, we have $D^B_X = {4 \over 3} Y_1 + {4 \over 3} Y_2 + {2 \over 3} X_1$.
    
    \begin{figure*}[h]
        \begin{tikzpicture}[
    dot/.style={circle, fill, inner sep=1.5pt}, 
    label distance=2pt
]
    \coordinate (A) at (0,0);
    \coordinate (B) at (6,0);
    \draw[thick] (A) -- (B);

    \node[dot, label=above:{$2 X_1$}] at (A) {};
    \node[dot, label=above:{$2 Y_1 + 2 Y_2$}] at (B) {};

    \node[dot, label=above:{$-\cano_X$}] (M) at ($(A)!0.5!(B)$) {};

    \node[dot, label=below:{$D^B_X$}] (T) at ($(A)!0.6667!(B)$) {};

\end{tikzpicture}
        \caption{Rational linear series $\Qls{{-\cano_X}}^B$.}\label[figure]{fig:P1xP1}
    \end{figure*}

    In this case, we have the inequality $\alpha_1 = {4 \over 3} > 1 = m_1$.
\end{example}

\printbibliography
\end{document}